\documentclass[
a4paper,
eqnright, reqno]{amsart}
\usepackage{graphicx}
\usepackage[backrefs
]{amsrefs}

\usepackage{hyperref}
\makeatletter
  \providecommand{\loglike}[1]{\mathop {\operator@font #1}\nolimits}
\makeatother
    \newtheorem{thm}{Theorem}[section]

    \newtheorem{cor}[thm]{Corollary}
    
    \theoremstyle{definition}

    \newtheorem{defn}[thm]{Definition}

    \theoremstyle{remark}
    \newtheorem{rem}[thm]{Remark}

\providecommand{\length}{\loglike{length}}
\providecommand{\Im}{\loglike{lm}}
\providecommand{\All}{\mathbb{U}}

\title{Isometric action of $SL_2(\mathbb{R})$ on homogeneous spaces}
\author{Anastasia V. Kisil}
\date{2nd October 2008. Revised: 7 Aug  2009}
\address{Trinity College, Cambridge, CB2 1TQ, England}
\email{ak528@cam.ac.uk}

\begin{document}
\begin{abstract}
We investigate the \(SL_2(\mathbb{R})\) invariant geodesic curves  with the associated
  invariant distance function in parabolic
  geometry. 
  Parabolic geometry naturally occurs in the study of
  \(SL_2(\mathbb{R})\) and is placed in between the elliptic and
  the  hyperbolic (also known as the Lobachevsky half-plane and 2-dimensional
  Minkowski half-plane space-time) geometries. Initially we attempt to use
  standard methods of finding geodesics but they lead to degeneracy in
  this setup.  Instead, by studying closely the two related elliptic
  and hyperbolic geometries we discover a unified approach to a more
  exotic and less obvious parabolic case. With aid of common
  invariants we describe the possible  distance functions that turn
  out to have some unexpected, interesting properties.
\end{abstract}

\maketitle

\section{Introduction}
\label{sec:introduction}
In this paper we will explore the isometric action of the semi-simple Lie group
 \(G=SL_2(\mathbb{R})\) (\(2 \times 2\) real matrices with determinant
 one) on the homogeneous spaces
\(G/H\) where \(H\) is one dimensional subgroup of \(G\). There are only
three such subgroups up to conjugacy, proved in~\cite{Lang85}*{\S~III.1}:
\begin{eqnarray}
\label{eq:rotation}
K&=& \left\{
  \begin{pmatrix}
    \cos\theta & -\sin\theta \\
    \sin\theta &  \cos\theta\\
  \end{pmatrix},
  \qquad
  \text{where}
  \quad
  0\le \theta<2 \pi \right\} \nonumber\\
N'&=& \left\{
  \begin{pmatrix}
    1 & 0 \\
    t & 1 \\
  \end{pmatrix},
  \qquad
  \text{where}
  \quad
  t \in \mathbb{R} \right\} \\
A'&=& \left\{
  \begin{pmatrix}
    \cosh\alpha & \sinh\alpha \\
    \sinh\alpha &  \cosh\alpha\\
  \end{pmatrix}, 
  \qquad
  \text{where}
  \quad \alpha \in \mathbb{R}\right\} \nonumber.
\end{eqnarray}
We can represent the action of \(SL_2(\mathbb{R})\) on the homogeneous
spaces \(G/H\) by M\"obius transformations:
  \begin{equation}
\label{eq:action_g}
    g \cdot w=\frac{aw+b}{cw+d} \quad \text{where}\quad w \in
  \mathbb{\All ,}\quad
   g=
  \begin{pmatrix}
    a & b\\
    c& d
  \end{pmatrix}
\in SL_2(\mathbb{R}).
\end{equation}
This is a group homomorphism from \(SL_2(\mathbb{R})\) to M\"obius
transformations of the upper half-plane due to the fact that
composition of two such M\"obius maps is represented by the product of
two respective matrices.  Here \( \All\) are numbers of the form
\(w=x+iy\) with \(i^2= \sigma=-1, 0, 1\) and are called \emph{complex}, \emph{dual}
and \emph{double} numbers respectively see
\citelist{\cite{BocCatoniCannataNichZamp07}
  \cite{HerranzOrtegaSantander99a}}.

Those three subgroups give rise to \emph{elliptic}, \emph{parabolic}
and \emph{hyperbolic} geometries (abbreviated \emph{ EPH}). The name
comes about from the shape of the equidistant orbits which are
ellipses, parabolas, hyperbolas respectively. Thus the Lobachevsky
geometry is elliptic (not hyperbolic) in our terminology.

EPH are 2-dimensional Riemannian, non-Riemannian and pseudo-Riemannian
geometries on the upper half-plane (or on the unit disc).
Non-Riemannian geometries, see~\citelist{\cite{BocCatoniCannataNichZamp07}
  \cite{MikesBascoBerezovski08a} \cite{KuruczWolterZakharyaschev05}
  \cite{BekkaraFrancesZeghib06}}, are a growing field with geometries
like Finsler in~\cite{Chern96a} gaining more influence. The Minkowski
geometry is formalised in a sector of a flat plane by means of double
numbers~\cite{BocCatoniCannataNichZamp07}, see also~\cite{Kisil04b}. 

Subgroups \(H\) from~\eqref{eq:rotation} fix the imaginary unit \(i\)
under the action of~\eqref{eq:action_g} and thus are known as
\emph{EPH rotations} (around \(i\)).  Consider a distance function invariant under
the \(SL_2(\mathbb{R})\) action. Then the orbits of \(H\) will be
equidistant points from \(i\), giving some indication on what the
distance function should be. But this does not determine the distance
entirely since there is freedom in assigning values to the orbits.
Review a well-known standard definition of distance \(d: X \times X
\to \mathbb{R^+}\) with:
\begin{enumerate}
\item \(d(x, y) \ge 0\),
\item \(d(x, y)=0\)  iff  \(x=y\),
\item \(d(x, y)=d(y, x)\),
\item \(d(x,y)\le d(x,z)+d(z,y)\),
\end{enumerate}
for all \(x\), \(y\), \(z\) \(\in X\). Although adequate in many cases,
the defined concept does not cover all interesting distance functions.
Examples include the recent study of distances with
omission of symmetry or  the triangle 
inequality as in ~\cite{KuruczWolterZakharyaschev05}. In the hugely important
Minkowski space-time the reverse of the triangle inequality holds.
We will be referring to this later.

Our consideration will be based on equidistant orbits, which
physically correspond to wavefronts with a constant velocity.  For
example, if you drop a stone in the pond the ripples you see will be
waves, which travelled the same distance from a dropping point.  A
dual description to wavefronts uses rays---the paths along which waves
travels, i.e. the geodesics in the case of a constant velocity.  The
duality between wavefronts and rays is provided by Huygens' principle
in~\cite{Arnold91}*{\S~46}.

Geodesics also play a central role in differential geometry
generalising the notion of a straight lines.  They are closely related
to a distance function: geodesics are often defined as curves which
extremize distance. As a consequence, along geodesics the distance
function is additive.



In the next section we will describe the invariant metrics in EPH
and the Riemannian approach to geodesics. In
Section~\ref{sec:algebraic-invariants} we describe all invariant
distance functions satisfying a mild assumption. In
Section~\ref{sec:additivity} we deduce geodesics from invariant
distances using the additivity property. An alternative construction
of invariant distances from invariant geodesics is described in
Section~\ref{sec:geometric-invariants}. The triangle inequality for
those distance functions is studied in
Section~\ref{sec:prop-dist-funct}. 

\section{Metric, Length and Extrema}
\label{sec:metr-length-extr}

Recall the established procedure of constructing geodesics in
Riemannian geometry (two-dimensional case) as in~\cite{Wilson08a}*{\S~7}:

\begin{enumerate}
\item Define the metric of the space: \(Edu^2+Fdu dv +Gdv^2\).
\item Define length for a curve \(\Gamma\) as:
\begin{equation}
\length(\Gamma)=\int_\Gamma (Edu^2+Fdu dv +Gdv^2)^{\frac{1}{2}}.
\end{equation}
\item Then geodesics will be defined as the curves which give a
  stationary point for length. 
\item Lastly the distance between two points is the length of a geodesic
  joining those two points.
\end{enumerate}

In this respect, to obtain the \(SL_2(\mathbb{R})\) invariant
(refereed simply as invariant in the rest of the
paper) distance we require the invariant metric. 

\begin{thm}
 The invariant metric in EPH cases is
 \begin{equation}
   d\label{eq:metric}
   s^2=\frac{du^2-\sigma dv^2}{v^2},
 \end{equation}
where \(\sigma=-1,0,1\) respectively. 
\end{thm}
In the proof below we will follow the same procedure as in~\cite{Carne06}*{\S~10}.
\begin{proof}
  In order to calculate the metric consider the
  subgroups~\eqref{eq:rotation} of M\"obius transformations that fix
  \(i\).  Denote an element of those rotations by \(E_\sigma\).  We
  require an isometry so:
  \begin{equation}
d(i,i+\delta v)=d(i,E_\sigma(i+\delta v)).
  \end{equation}
Using the Taylor series we get:
\begin{equation}
E_\sigma(i+\delta v)=i+J_\sigma(i)\delta v+o(\delta v),
\end{equation}
where the Jacobian denoted \(J_\sigma\) respectively is: 
\begin{equation}
  \begin{pmatrix}
\cos2\theta & -\sin2\theta \\
\sin2\theta &  \cos2\theta\\
  \end{pmatrix}
,
  \qquad
  \begin{pmatrix}
1 & 0 \\
2t & 1 \\
  \end{pmatrix}
\quad
\text{or}
  \quad
  \begin{pmatrix}
\cosh2\alpha & \sinh2\alpha \\
\sinh2\alpha &  \cosh2\alpha\\
  \end{pmatrix}.
\end{equation}
A metric is invariant under the above rotations if it is preserved under the
linear transformation:
\begin{equation}
  \begin{pmatrix}
    dU \\
    dV \\
  \end{pmatrix}
  =J_\sigma
  \begin{pmatrix}
    du\\
    dv\\
  \end{pmatrix},
\end{equation}
which turns out to be, \(du^2-\sigma dv^2\) in the three cases.

To calculate the metric at an arbitrary point \(w=u+iv\) we map \(w\) to
\(i\) by an affine M\"obius  transformation, which acts transitively
on the upper half-plane
\begin{equation}\label{eq:s-map}
  r^{-1}: \quad w \to \frac{w-u}{v}
\end{equation}
hence there is a factor of \((\frac{1}{v})^2\) multiplying the metric
giving \(ds^2=\frac{du^2-\sigma dv^2}{v^2}\).
\end{proof}
\begin{cor}With the notation from above, for an arbitrary curve \(\Gamma\):
\begin{equation}
\length(\Gamma)=\int_\Gamma \frac{(du^2-\sigma
  dv^2)^{\frac{1}{2}}}{v}.
\end{equation}
\end{cor}

In the two non-degenerate cases (elliptic and hyperbolic) to find the
geodesics is straightforward, it is now the case of solving the
Euler-Lagrange equations and hence finding the minimum or the maximum
respectively. The Euler-Lagrange equations for the metric
\eqref{eq:metric} take the form:
\begin{equation}
\label{eq:EL}
  \frac{d}{d t} \left( \frac{\dot{\gamma_1}}{y^2}\right)=0, \qquad
  \frac{d}{d t}\left( \frac{\sigma
    \dot{\gamma_2}}{y^2}\right)=\frac{\dot{\gamma_1}^2-\sigma
    \dot{\gamma_2}^2}{y^3}.
\end{equation}
where \(\gamma\) is a smooth curve \(\gamma(T)=(\gamma_1(T),\gamma_2(T))\) and \( T
\in (a,b)\).

For \(\sigma=-1\) the solution is well-known: semicircles orthogonal
to the real axes or vertical lines, as in~\cite{Beardon05a}*{\S~15}.
Equations of the ones passing though \(i\) are, see~\cite{Wilson08a}*{\S~7}:
\begin{equation}
(x^2 +y^2)\sin 2t -2x\cos 2t -\sin2 t =0
\end{equation}
where \(t \in \mathbb{R}\). And the distance function is then:
  \begin{equation}
\label{eq:elliptic}
    d(z,w)=\sinh^{-1}\frac{\left | z-w \right |}{2\sqrt{\Im[z]\Im[w]}}
  \end{equation}
where \(\Im[z]\) is the imaginary part of \(z\).

In the hyperbolic case when \(\sigma=1\) there are two families of
solutions,  one space-like, one time-like:
\begin{equation}
x^2-y^2 -2tx +1=0 \quad \text{and} 
\end{equation}
\begin{equation}
 (x^2 -y^2)\sinh 2t-2x\cosh2t +\sinh 2t=0
\end{equation}
with \(t \in \mathbb{R}\). Those are again orthogonal
see  \citelist{\cite{Kisil06a} \cite{BocCatoniCannataNichZamp07}} to the real axes. And the distance
functions are:
\begin{equation}
  \label{eq:hyperbolic}
  d(z,w)= \left\{ \begin{array}{ll}
     2 \sin^{-1}\frac{\sqrt{\left | \Re{[z-w]}^2-\Im{[z-w]}^2 \right | }}{2\sqrt{\Im[z]\Im[w]}} ,&   \textrm{when space-like;} \\[12pt] 
     2 \sinh^{-1}\frac{\sqrt{\left |\Re{[z-w]}^2-\Im{[z-w]}^2\right |}}{2\sqrt{\Im[z]\Im[w]}} ,&  \textrm{when time-like,}
    \end{array} \right.
\end{equation}
where \(\Re[z]\) and \(\Im[z-w]\) are the real and imaginary part of \(z\).

But in the parabolic framework the only solution of \eqref{eq:EL} are
vertical lines, as in~\cite{Yaglom79}*{\S~3}. Note that they are again
orthogonal to the real axes. They indeed minimise the distance between
two points \(w_1\), \(w_2\) since the geodesic is up the line
\(x=\Re[w_1]\) through infinity and down \(x=\Re[w_2]\). Any points on
the same vertical lines have distance zero, so \(d(w_1, w_2)=0\) for
all \( w_1\), \(w_2\) which is a very degenerate function. Hence we go
on to study further the algebraic and geometric invariants to find a
more adequate answer.

\begin{rem}
The same geodesic equations can be obtained by Beltrami's method
\cite{BocCatoniCannataNichZamp07}
\end{rem}
\section{Algebraic Invariants}
\label{sec:algebraic-invariants}

We seek all real valued functions invariant under the group action of
\(SL_2(\mathbb{R})\)~\eqref{eq:action_g}:
\begin{displaymath}
  f(g(z),g(w))=f(z,w) \quad \textrm{for all} \quad z,w\in \All \quad
  \textrm{and} \quad g\in SL_2(\mathbb{R})
\end{displaymath}
where \(\All\) is the complex, dual or double numbers.  One such
function is, similarly to  \eqref{eq:elliptic}~\eqref{eq:hyperbolic}:
\begin{equation}
  F(z,w)=\frac{\left | z-w \right |_\sigma}{\sqrt{\Im[z]\Im[w]}},
\end{equation}
which can be shown by a simple direct calculation. Note that
\begin{equation}
 \left | z-w \right |_\sigma^2=\Re{[z-w]}^2-\sigma \Im{[z-w]}^2
\end{equation}
in analogy with the metric in EPH geometries, similarly to what is
done in~\cite{Yaglom79}*{App.~C}. We will
need the following definition:
\begin{defn}
\label{th:monotonious}
A function \(f : X \times X \to \mathbb{R^+}\) is called a
\emph{monotonous} distance if \(f(\Gamma(0), \Gamma(t))\) is a continuous
monotonically increasing function of \(t\) where
\(\Gamma~:~[0,1)~\to~X\) is a smooth curve with \(\Gamma(0)=z_0\) that
intersects all equidistant orbits of \( z_0\) exactly once.
\end{defn}
For example function \(F(z,w)\) is monotonous.

\begin{thm} \label{th:monotonic} 
 A  monotonous function \(f(z,w)\) is
  invariant under \(g \in SL_2(\mathbb{R}) \) if and only if there exists a monotonically
  increasing continuous real function \(h\) such that \(f(z, w)=h
  \circ F(z,w)\).
\end{thm}
\begin{proof}
  Given \(f(z, w)=h \circ F(z,w)\) then:
  \begin{displaymath}
    f(g(z),g(w))=h\circ F(g(z),g(w))=h \circ F(z,w)=f(z,w)
  \end{displaymath}
  with  \(g \in SL_2(\mathbb{R})\). Also \(F(z,w)\) is monotonous and
  so \(h \circ F(z,w)\) is.
  
  Conversely, suppose there exists another function with such a property say,
  \(H(z,w)\).  Due to invariance under \(SL_2(\mathbb{R})\) this can
  be viewed as a function in one variable if we apply \(r^{-1}\)
  (cf.~\eqref{eq:s-map}) which sends \(z\) to \(i\) and \(w\) to
  \(r^{-1}(w)\).  Now by considering a fixed smooth curve \(\Gamma\)
  from \ref{th:monotonious} we can completely define \(H(z,w)\) as a
  function of a single real variable \(h(t)=H(i, \Gamma(t))\) and
  similarly for \(F(z,w)\):
  \begin{displaymath}
    H(z,w)=H(i,r^{-1}(w))=h(t) 
    \quad\text{and}\quad
    F(z,w)=F(i,r^{-1}(w))=f(t)
  \end{displaymath}
  where  \(h\) and \(f\) are both continuous and  monotonically
  increasing  since they represent distance.
  Hence the inverse \(f^{-1}\) exists
  everywhere by the inverse function theorem. So:
  \begin{eqnarray}
    H(i,r^{-1}(w))&=&h \circ f^{-1} \circ F(i,r^{-1}(w)).
  \end{eqnarray}
  Note that \(hf^{-1}\) is monotonic as its the composition of two
  monotonically increasing functions and this ends the
  proof.
\end{proof}
\begin{rem}
  The above proof carries over to a more general
  theorem stating: If there exist two monotonous functions \(F(x,y)\)
  and \(H(x,y)\) invariant under a transitive action of the group G
  then  there exists  a monotonically increasing real function
  \(h\)  such that \(H(z,w)=h \circ F(z,w)\).
\end{rem}
As discussed in the previous section, in elliptic and hyperbolic
geometries the function \(h\) from above  is either \(\sinh^{-1}t\) or
\(\sin^{-1}t\) \eqref{eq:elliptic}~\eqref{eq:hyperbolic}. Hence it
is reasonable to try inverse trigonometric and hyperbolic
functions in the intermediate parabolic case too.
\begin{rem}
  The above result sheds light on the freedom we have;  we can either
  (i) ``label'' the equidistant orbits with numbers i.e choose a function \(h\) which
  will then determine the geodesic; (ii) or
  choose a geodesic which will then determine  \(h\). Those two 
  approaches  are reflected in the next two chapters.
\end{rem}

\section{Additivity}
\label{sec:additivity}

As pointed out, earlier there might not be a distance function which
satisfies all the traditional properties. But we still need the key
ones, in light of this we make the following definition: 
\begin{defn}
  \emph{Geodesics} are smooth curves along which the distance function is additive. 
\end{defn}
\begin{rem}
 It is important that this definition is relevant in all EPH cases,
 i.e. in the elliptic and hyperbolic cases it would produce the
 well-known geodesics defined by the extremality condition. 
\end{rem}
Schematically the proposed approach is:
\begin{equation}
  \text{invariant distance} \quad  \xrightarrow{additivity}
  \quad \text{invariant geodesic} 
\end{equation}
compare this with the Riemannian:
\begin{equation}\label{eq:metr-geod-dist}
  \text{local metric} \quad  \xrightarrow{extrema} \quad
  \text{geodesic} \quad  \xrightarrow{integration} \quad
  \text{distance}. 
\end{equation}
Let us now proceed with  finding geodesics from a distance function.

Suppose given a function \(d(w, w')\) with \( w\), \(w' \in \All\) there
exists a smooth curve joining every two points along which \(d\) is
additive. We can view the function \(d(w,w')\) as a real function in
\(4\) variables say \(f(u, v, u', v')\).  Take two points
\(w_1=u_1+iv_1\) and \(w_2=u_2+iv_2\) on such a curve.  Then consider
a nearby point \(w_3=w_2+\delta w=u_2+\delta u +i(v_2+\delta v)\) and write the
additivity condition:
\begin{equation}
\label{eq:add}
  d(w_1,w_2)+d(w_2,w_3)=d(w_1, w_3).
\end{equation}
Use the Taylor series, viewing \(f\) as a function of four real
variables, we obtain:
\begin{equation}
  d(w_2,w_3)=f(u_2, v_2, u_2+\delta u,v_2+\delta v)=f(u_2, v_2,u_2,
  v_2)+J_f(u_2, v_2,u_2,
  v_2) \delta w+o(\delta w)
\end{equation}
where \(J_f\) is the Jacobian.
Since \(f\) is a distance function we demand that distance from a
point to itself is zero hence \(f(u_2, v_2,u_2,
v_2)=0\). Also:
\begin{equation}
  d(w_1,w_3)=f(u_1, v_1, u_2+\delta u,v_2+\delta v)=f(u_1, v_1,u_2,
  v_2)+J_f(u_1, v_1,u_2,
  v_2) \delta w+o(\delta w),
\end{equation}
substituting this into \eqref{eq:add} gives:
\begin{eqnarray*}
  J_f(u_2, v_2,u_2,
  v_2) \delta w &=& J_f(u_1, v_1,u_2,
  v_2) \delta w, \\
  f_3(u_2,v_2,u_2,v_2) \delta u +f_4(u_2, v_2, u_2, v_2) \delta v &=&
  f_3(u_1, v_1, u_2,v_2) \delta u
  +f_4(u_1, v_1, u_2,v_2) \delta v.
\end{eqnarray*}
And finally obtain the differential equation:
\begin{equation}
\label{eq:differential}
  \frac{\delta v}{\delta u} = \frac{-f_3(u_1, v_1, u_2,v_2)
    +f_3(u_2,v_2,u_2,v_2)}{f_4(u_1, v_1, u_2,v_2) -f_4(u_2,v_2,u_2,v_2)}, 
\end{equation}
where \(f_n\) stands for partial derivative with respect to the
\(n\)-th variable.

A natural choice for  \(d(w,w')\) is
\(\sin_{\breve{\sigma}}^{-1}\frac{\left|
  w-w' \right|_\sigma}{2\sqrt{\Im[w]\Im[w']}}\) where:
\begin{equation}
  \sin_{\breve{\sigma}}^{-1}t= \left\{ \begin{array}{ll}
      \sinh^{-1} t, & \textrm{if \(\breve{\sigma}=-1;\)}\\
      2t, & \textrm{if \(\breve{\sigma}=0\)};\\
      \sin^{-1}t,  & \textrm{if \(\breve{\sigma}=1\)}.
    \end{array} \right.
\end{equation}
known as elliptic, parabolic and hyperbolic inverse sine, for
motivation see \citelist{\cite{Kisil07a} \cite{HerranzOrtegaSantander99a}}. Note that
\(\breve{\sigma}\) is entirely different from \( \sigma\) although it
takes the same values. It is used to denote the possible sub-cases
within the parabolic geometry alone. Substituting those functions in
\eqref{eq:differential} gives:
\begin{equation}
  \frac{\delta v}{\delta u}=\frac{2v_2}{u_2-u_1}
  -\frac{\sqrt{\left|\breve{\sigma}{(u_1-u_2)}^2+4v_1v_2 \right|}}{u_2-u_1},
\end{equation}
where \(\breve{\sigma}=-1,0,1\) correspond to elliptic, parabolic
and hyperbolic inverse sine of the distance function.  Now taking
\(u_1=0\) and \(v_1=1\) and varying \(u_2+iv_2\) gives the differential
equation:
\begin{equation}
  \frac{\delta v}{\delta u}=\frac{2v}{u}-\frac{\sqrt{\left|
        \breve{\sigma} {u}^2+4v \right|} }{u}.
\end{equation}
Solutions are families of parabolas
\((\breve{\sigma}+4t^2)u^2-8tu-4v+4=0\). This is a subset of a very
important invariant class of curves called \emph{cycles} defined
in~\cite{Yaglom79}*{\S~7}, and further studied in
\cite{Kisil06a}%
. Cycles could be thought of the natural curves in those geometries,  in the
EPH cases they are circles, parabolas and hyperbolas respectively.
Summarising:
\begin{thm}
  The geodesics through \(i\) defined with the distance function
  \(\sin_{\breve{\sigma}}^{-1}\frac{\left |z-w
    \right |}{2\sqrt{\Im[z]\Im[w]}}\) are parabolas of the form
  \((\breve{\sigma}+4t^2)u^2-8tu-4v+4=0\).
\end{thm}

\section{Geometric Invariants}
\label{sec:geometric-invariants}

As we discussed earlier the invariant metric alone within Riemannian
framework produces only the trivial distance function.  In this
section we work out the distance from the metric and geodesics.
Schematically:
\begin{equation}
  \text{invariant metric \(+\) invariant geodesics} \quad  \xrightarrow{integration} \quad
  \text{distance function}.
\end{equation}
Again looking back at~\eqref{eq:metr-geod-dist} we can see that those
two approaches are similar.

Geodesics should form an invariant subset of an invariant class of
curves with no more than one curve joining every two points. Here this
class is cycles as they are the most and almost the only key objects
in EPH, shown in~\cite{Kisil05a}. Such a subset may be characterised
by an invariant algebraic condition, which in analogy with elliptic
and hyperbolic geometries is taken to be focal orthogonality
(f-orthogonality) to the real axes defined in
\cite{Kisil06a}*{\S~4.3}. In brief a cycle is f-orthogonality to the
real axes if the real axes inverted in a cycle is orthogonal (in the
usual sense) to the real axes.  Explicitly parabola
\(ku^2-2lu-2nv+m=0\) is f-orthogonal to the real axes if:
\begin{equation}
  l^2+\breve{\sigma}n^2-mk=0,
\end{equation}
where \(\breve{\sigma}=-1\), \(0\), \(1\) similarly to (but
independently of!) \(\sigma\). 

As a starting point consider the cycles that pass though \(i\). It is
enough to specify only one such f-orthogonal cycle; the rest will be
obtained by M\"obius transformations fixing \(i\) i.e parabolic
rotations \eqref{eq:rotation}.  Within those constraints there are
three different families of parabolas. They are obtained by parabolic
rotation of principal parabolas:
\begin{equation}
\label{eq:principle}
  y=\frac{\breve{\sigma}}{4}x^{2}+1,
\end{equation}
where
\(\breve{\sigma}=-1,0,1\). Explicitly they are given by equations:
\(\breve{\sigma} +4t^2x^2-8tx+4=4y\). Note  that
those are exactly the same geodesics as in the previous section. Hence
we know what the distance function has to be. But we still give a
sketch of how to work out the distance below. This calculation does
not involve anything from the previous section and is in a way more
elementary and intuitive.

\begin{rem}
  The length taken along an arbitrary parabola \(ax^2+bx+c=y\)
  (parametrised as \(t+i(at^2+bt+c)\)) using the standard definition of
  length is:
\begin{equation}
  \length(\Gamma)=\int_\Gamma \frac{dt}{at^2+bt+c}.
\end{equation}
Depending on whether the discriminant of the denominator is positive,
zero or negative the results are trigonometric, rationals or hyperbolic
functions respectively. This gives an inside to why there are only
three distinct types of geodesics.
\end{rem}
Consider specifically the f-orthogonal parabolas described above and
use a trick of moving one point to \(i\) and the second on to the
principle parabola as in  \eqref{eq:principle}~\cite{Wilson08a}*{\S~7}. Then
the distance along the principle parabola is:
\begin{equation}
  \int\limits_0^x\frac{dt}{\frac{1}{4}\breve{ \sigma}t^{2}+1} =\left\{ \begin{array}{ll}
  4\log \frac{2+x}{2-x},  & \textrm{if \(\breve{\sigma}=-1\)};\\ 
      x,  & \textrm{if \(\breve{\sigma}=0\)};\\
        \tan^{-1} \frac{x}{2},  & \textrm{if \(\breve{\sigma}=1\)}.
    \end{array} \right.
\end{equation}
Given an arbitrary point \( (u, v)\) the distance to \(i\) can be
calculated by finding where the orbit of \((u, v)\) intersects the
principle parabola, which is a matter of solving simultaneous
equations:
\begin{equation}
  v= \frac{\breve{\sigma}}{4}u^{2}+1, \qquad
  v=ku,
\end{equation}
where \(k=\frac{v_0}{u_0}\) giving:
\begin{equation}
  u=\frac{u_0}{ \sqrt {v_0-\frac{\breve{\sigma}u_0^2}{4}}}.
\end{equation}
Finally the  distance from \((u_1,v_1)\) to \((u_2,v_2)\) is calculated by
applying the M\"obius transformation \(r^{-1}(w)\)
(cf.~\eqref{eq:s-map}) which
sends \(w_1\) \(\mapsto\) \(i\) and \(w_2\) \(\mapsto\)
\( (\frac{u_2-u_1}{v_1},\frac{v_2}{v_1})\). So:
\begin{equation}
  d(w_1,w_2) = \left\{ \begin{array}{ll}
      8 \sinh^{-1} \frac{\left |u_2-u_1\right |}{2 \sqrt {v_1 v_2}}, & \textrm{if \(\breve{\sigma}=-1.\)};\\[5pt]
      \frac{\left |u_2-u_1\right |}{\sqrt {v_1 v_2}}, & \textrm{if \(\breve{\sigma}=0\)};\\[5pt]
 \sin^{-1}{ \frac{\left |u_2-u_1\right |}{2 \sqrt{v_1 v_2}}},   & \textrm{if \(\breve{\sigma}=1\)};\\
 \end{array} \right.
\end{equation}
Although the answer above is of no surprise, in light of previous
section, it is nevertheless gives more inside how \(\frac{\left|
    u_2-u_1\right |_\sigma}{2 \sqrt {v_1v_2}}\) and \( 
\sin_{\breve{\sigma}}^{-1}\) appears. Diffusion of the parabolic
geometry into three different sub-cases is known and got the names
\emph{\(P_e\), \(P_p\) and \(P_h\)} to stand for elliptic, parabolic
and hyperbolic flavours, for further examples see~\cite{Kisil07a}.
The geodesics have been drawn in Figure~\ref{fig:geodesic} and it is
striking how there seems to be a ``continuous'' transformation between
the geometries. We can see the transitions from the elliptic case to
\(P_e\) then to \(P_p\) to \(P_h\) to the hyperbolic space-like and
finally to light-like.

There is one more pleasant parallel between all the geometries. In the
Lobachevsky and Minkowski geometries the \emph{centre} of geodesics
lies on the real axes. In the parabolic geometry the respective
\emph{foci} of geodesic parabolas lie on the real axes. Here \(P_h\)
focus is just the usual focus of the parabola, the \(P_p\) is the
vertex of the parabola and the \(P_e\) is the point of directrix
nearest to the parabola. 
Those foci are closely linked~\cite{Kisil05a}*{\S~4.3} to the
f-orthogonality we used earlier.

\begin{figure}[htbp]
  \centering
  \includegraphics[scale=0.5,angle=0]{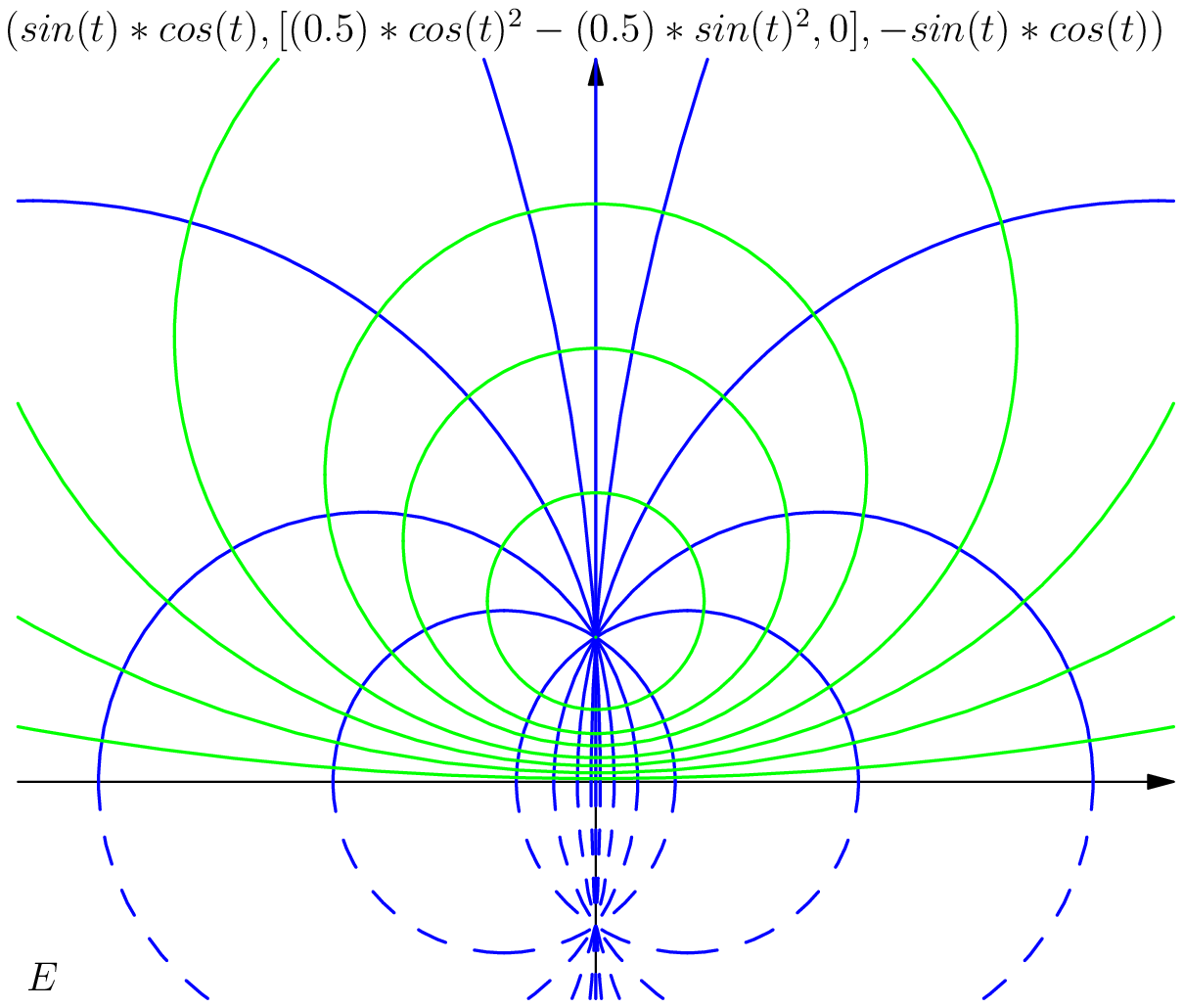}\hfill
 \includegraphics[scale=0.5,angle=0]{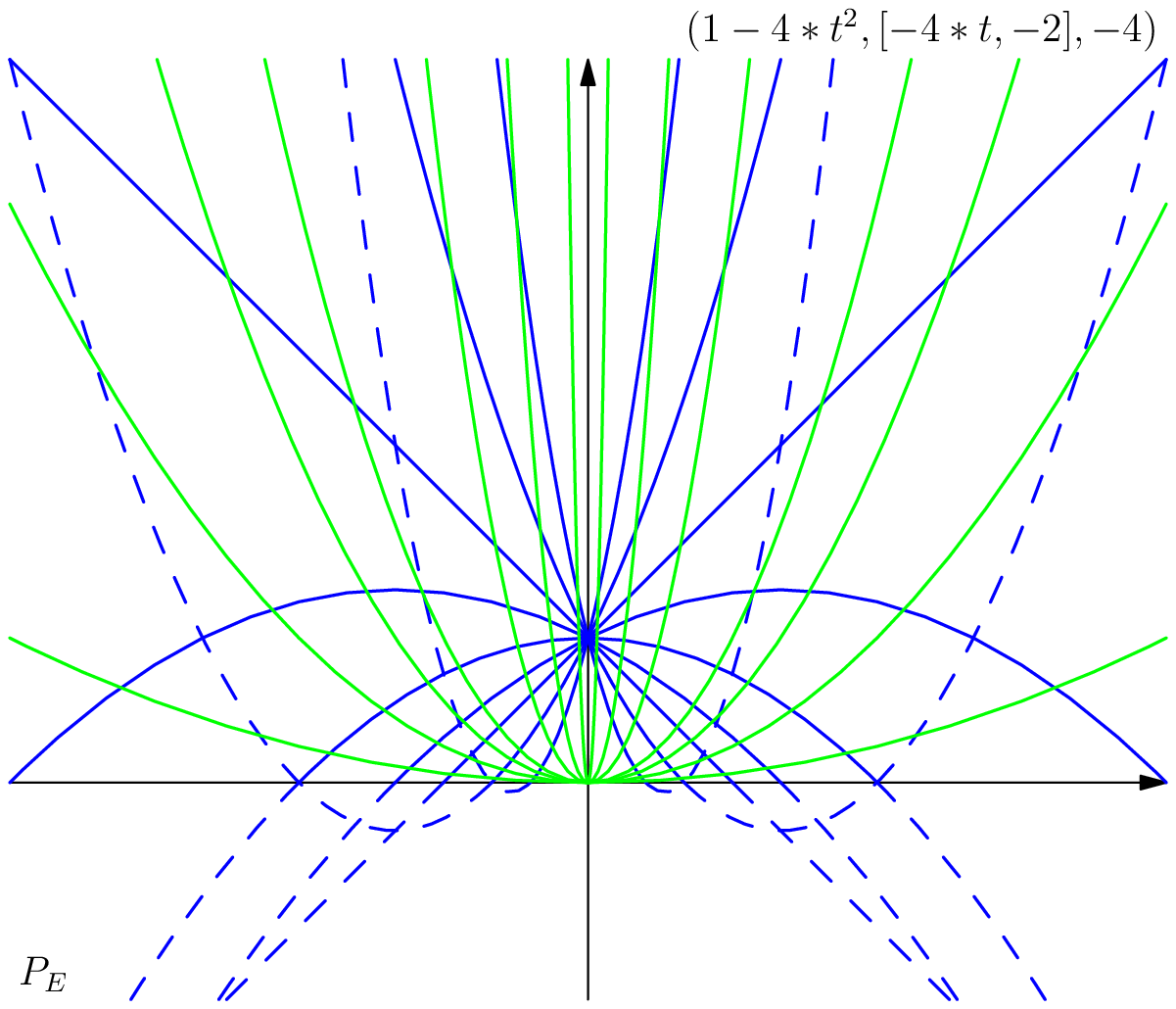}\\[10pt]
 \includegraphics[scale=0.5,angle=0]{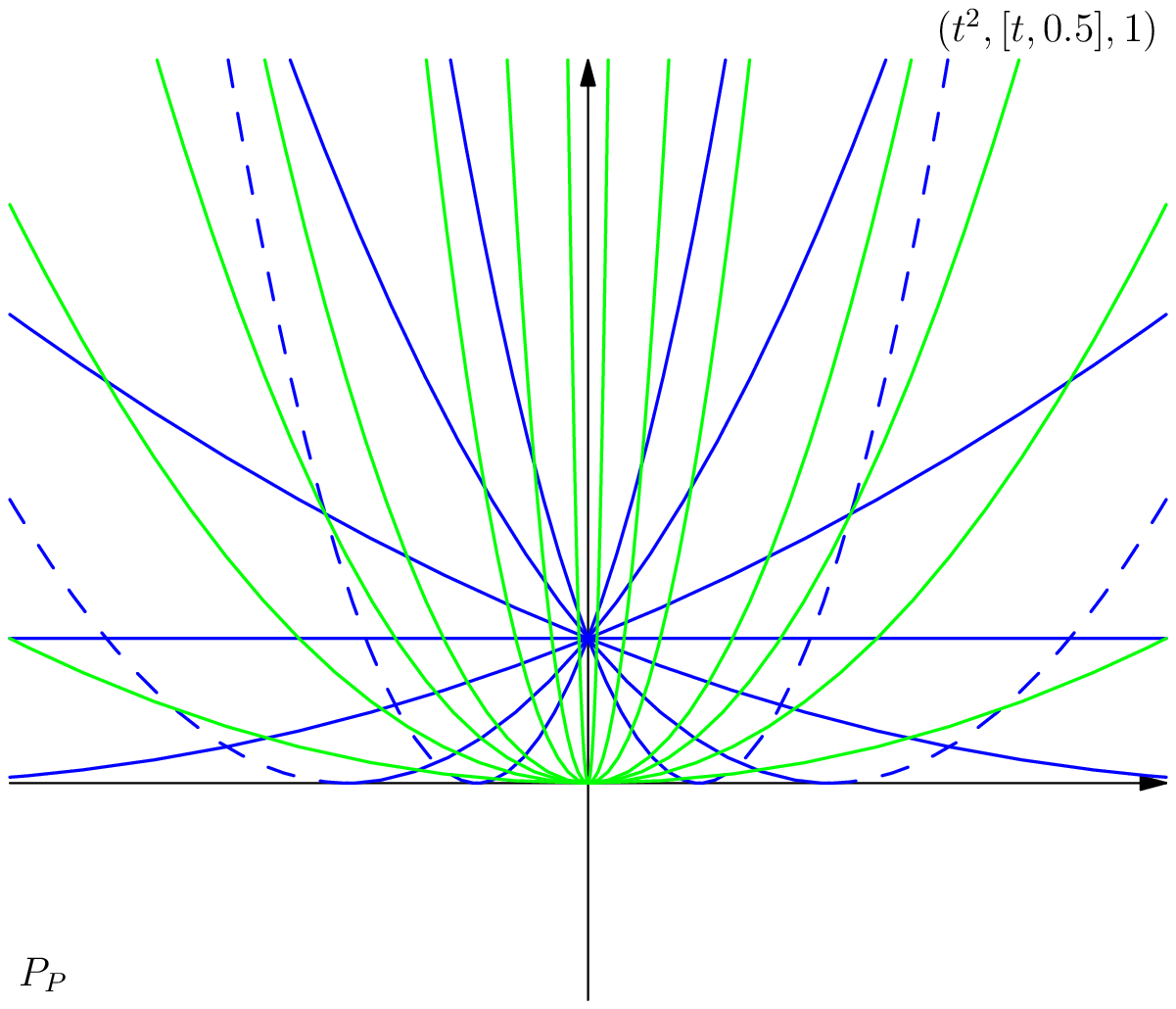}\hfill
  \includegraphics[scale=0.5,angle=0]{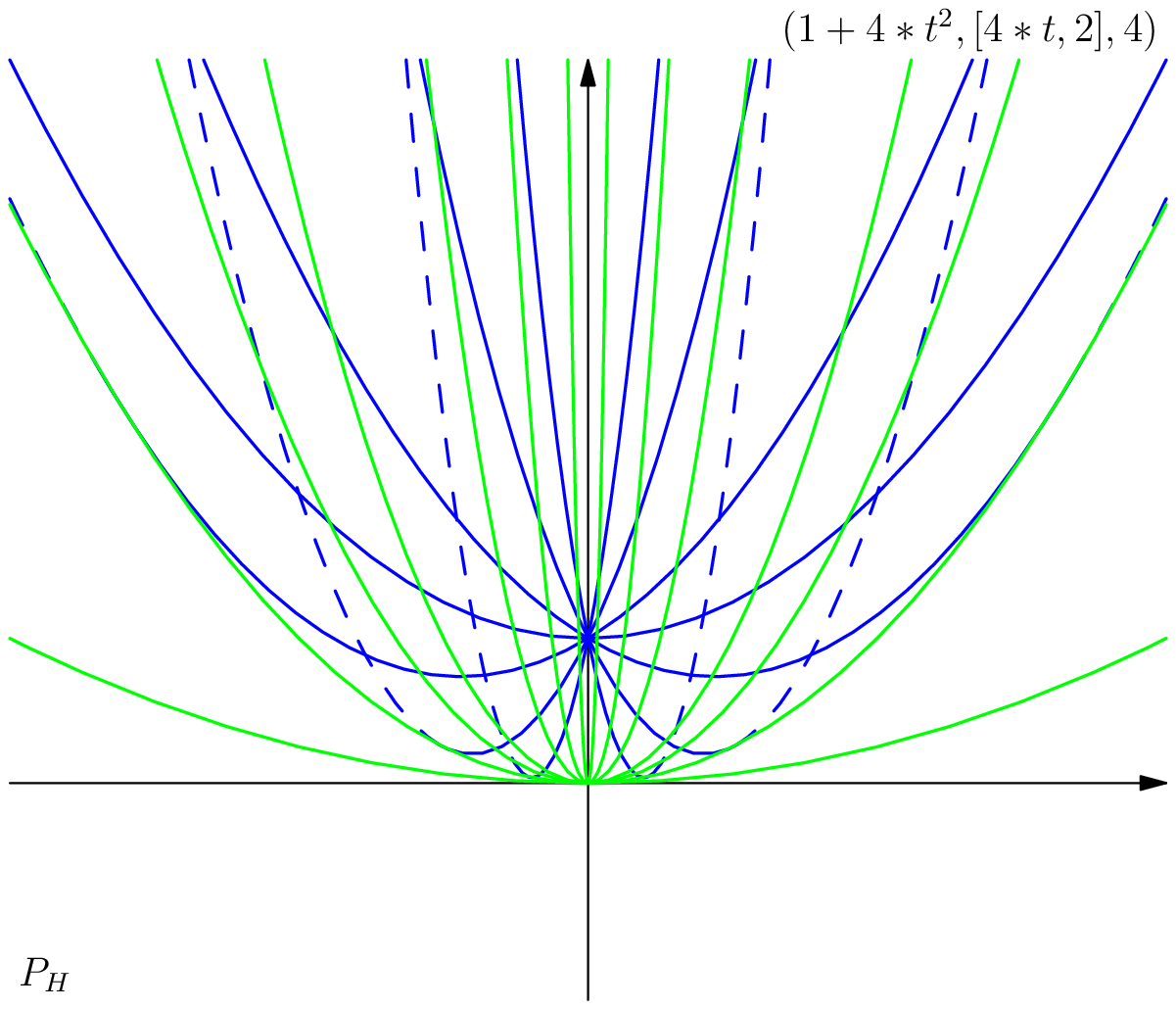}\\[10pt]
 \includegraphics[scale=0.5,angle=0]{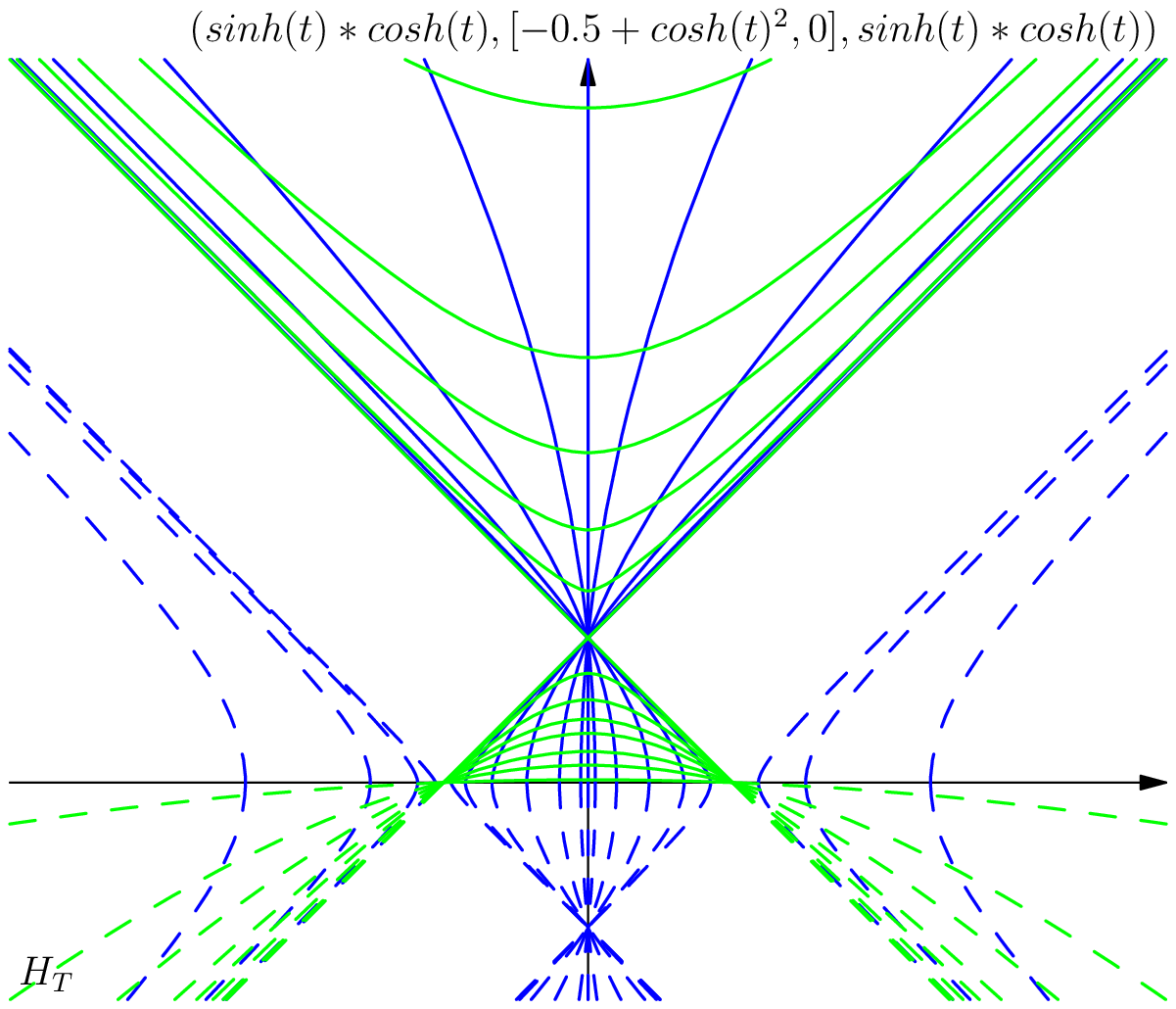}\hfill
 \includegraphics[scale=0.5,angle=0]{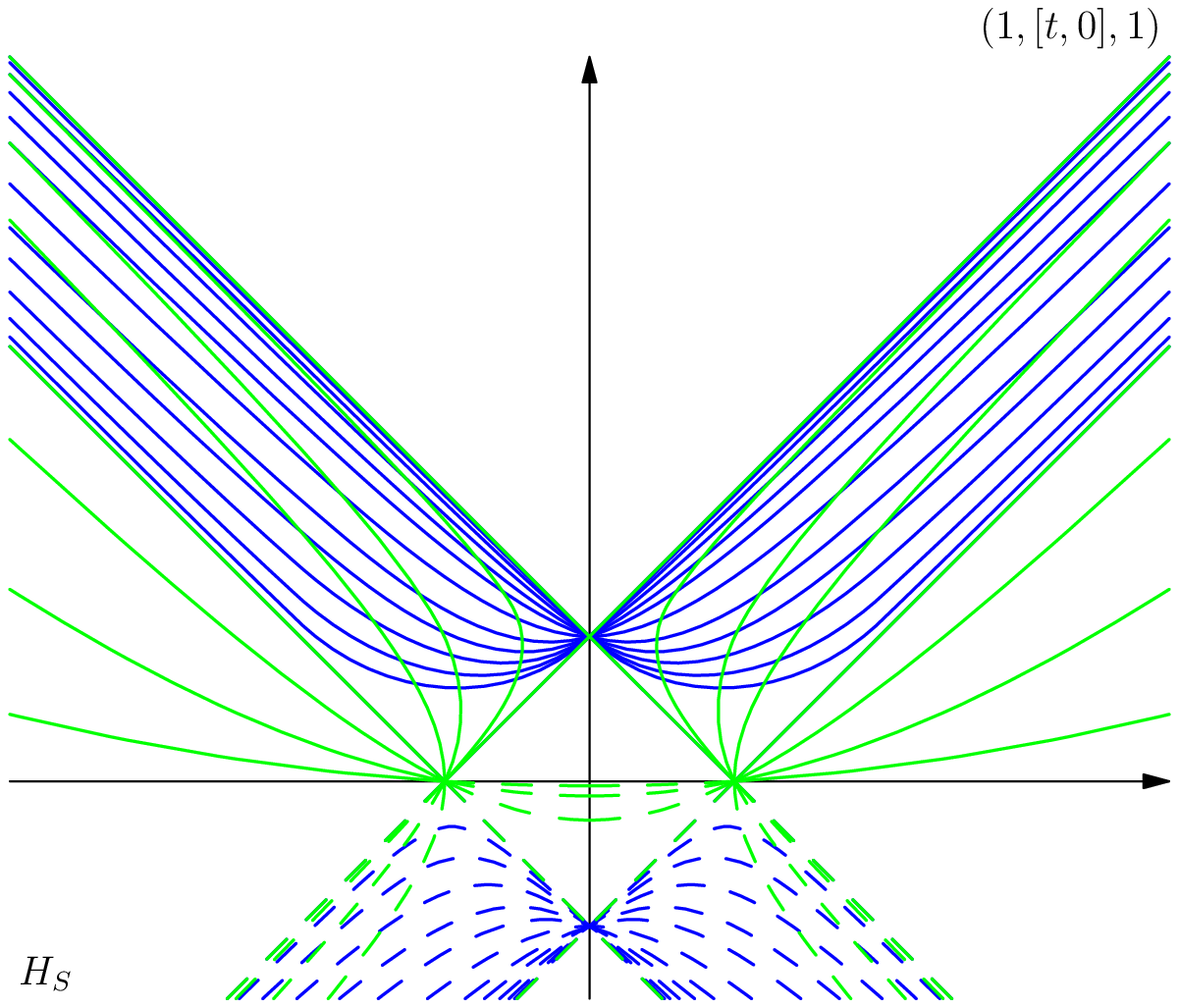}
  \caption{Showing geodesics (blue) and equidistant orbits
    (green) in EPH geometries. Above are written \((k,[l,n],m)\) in 
 \(kx^2-2lx-2nv+m=0\) giving the equation of geodesics.}
  \label{fig:geodesic}
\end{figure}

\begin{rem}
  We can translate all the above results from the upper half plane
  model we have been using  to the unit disc model. This is done via the
  Cayley transform which in parabolic geometries is given by, derived in~\cite{Kisil05a}:
  \begin{equation}
    \label{eq:cayley}
    w \mapsto \frac{2w-i}{i\breve{\sigma}w+2}
  \end{equation}
  Again note that the values of \(\breve{\sigma}=-1,0,1\) correspond
  to \(P_e\), \(P_p\) and \(P_h\) flavours of the parabolic unit
  discs. Then applying the inverse Cayley transform we can see
  that the invariant distance between two points \(u_1+iv_1\) and \(u_2+iv_2\)
  in the disc is:
  \begin{equation}
    \label{eq:disk-dist}
    \sin_{\breve{\sigma}}^{-1}  2\frac{\left |u_2-u_1\right |}{2
      \sqrt{(1+2v_1+\breve{\sigma}u_1^2)(1+2v_2+\breve{\sigma}u_2^2)}}. 
  \end{equation}
  An interesting feature of the Cayley transform~\eqref{eq:cayley} is
  as follows: in all three flavours of the parabolic unit disk the
  real line is a geodesic passing the origin for the
  respective distance~\eqref{eq:disk-dist} with the same value of
  \(\breve{\sigma}\) as in~\eqref{eq:cayley} .  
\end{rem}

\section{Properties of the distance function}
\label{sec:prop-dist-funct}

In the introduction, we listed properties of the standard distance
functions which we are going to re-visit in the context of obtained
invariant functions
\(d(z,w)_{\breve{\sigma}}=\sin_{\breve{\sigma}}^{-1}
\frac{\left|\Re[z-w] \right |}{2\sqrt{\Im[z]\Im[w]}}\).  Two of the
four properties hold: it is clearly symmetric and positive for every
two points. But the distance of any point to a point on the same
vertical line is zero so \(d(z, w)=0\) does not imply \(z=w\).  This
can be overcome by introducing a different distance function just for
the points on the vertical lines as is done in~\cite{Yaglom79}*{\S~3}.
Note that we still have \(d(z,z)=0\) for all \(z\).

Also the triangle inequality does not hold but an interesting variation on
it is true:

\begin{thm}
  Take any \(SL_2(\mathbb{R})\) invariant distance function and take
  two points \(w_1, w_2\) and the geodesic (in the sense of
  section~\ref{sec:additivity}) through points.  Consider the strip
  \(\Re[w_1]<u<\Re[w_2]\) and take a point \(z\) in it. Then the
  geodesic divides the strip into two regions where \(d(w_1,w_2)\le
  d(w_1,z)+d(z,w_2)\) and where \(d(w_1,w_2)\ge d(w_1,z)+d(z,w_2)\).
\end{thm}

\begin{rem}
  This is a kind of intermediate theorem between the elliptic case where
  \(d(w_1,w_2)\le d(w_1,z)+d(z,w_2)\) for all \( w_1\), \(w_2\), \(z \in \mathbb{C}\)
  and the hyperbolic geometry where the converse is true
  \(d(w_1,w_2)\ge d(w_1,z)+d(z,w_2)\). 
\end{rem}

\begin{proof}
  The only possible invariant distance function in parabolic geometry
  is of the form \(d(z, w)=h \circ \frac{\left |\Re[z-w] \right
    |}{2\sqrt{\Im[z]\Im[w]}}\) where \(h\) is a monotonically
  increasing continuous real function by Thm.~\ref{th:monotonic}. Fix
  two points \(w_1\), \(w_2\) and the geodesic though them. Now
  consider some point \(z=a+ib\) in the strip. The distance function
  is additive along a geodesic so \(d(w_1, w_2)= d(w_1,
  w(a))+d(w(a),w_2)\) where \(w(a)\) is a point on the geodesic with
  real part equal to \(a\). But if \(\Im[w(a)]<b\) then \(d(w_1,
  w(a))>d(w_1, z)\) and \(d(w(a), w_2)>d(z, w_2)\) which implies
  \(d(w_1,w_2)> d(w_1,z)+d(z,w_2)\).  Similarly if \(\Im[w(a)]>b\)
  then \(d(w_1,w_2)< d(w_1,z)+d(z,w_2)\).
\end{proof}
\begin{rem}
  The reason for the ease with which the result falls out is the fact
  that the distance function is additive along the geodesics. This
  justifies the definition of geodesic in terms of additivity. 
\end{rem}


\begin{figure}[htbp]
  \centering
 \includegraphics[scale=0.6,angle=0]{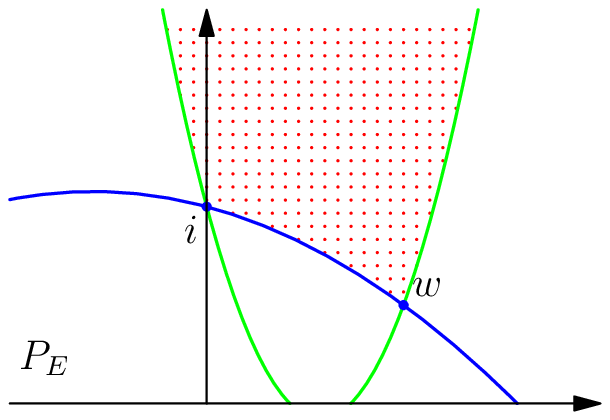}\hfill
 \includegraphics[scale=0.6,angle=0]{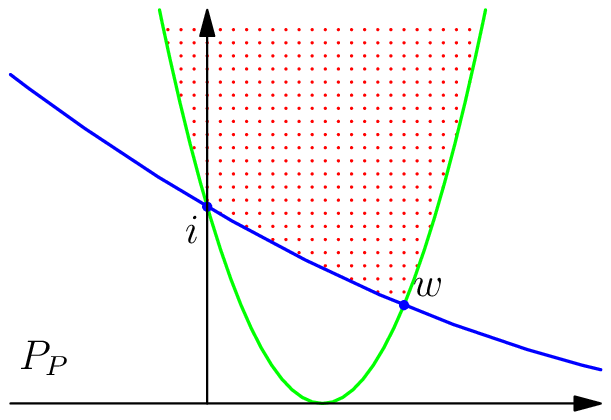}\hfill
\includegraphics[scale=0.6,angle=0]{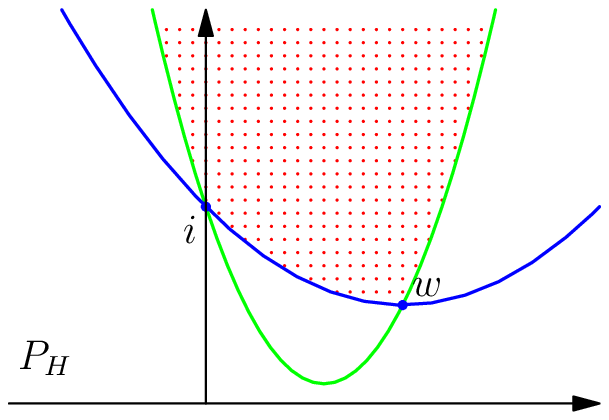}
  \caption{Showing the area where the triangular inequality
    fails (marked red).}
  \label{fig:triangle}
\end{figure}
To illustrate those ideas look at the region where the converse of the
triangular inequality holds for
\(d(z,w)_{\breve{\sigma}}=\sin_{\breve{\sigma}}^{-1}
\frac{\left|\Re[z-w] \right |}{2\sqrt{\Im[z]\Im[w]}}\) marked red on Figure
\ref{fig:triangle}. It is enclosed by two parabolas both of the form
\((\breve{\sigma}+4t^2)u^2-8tu-4v+4=0\) (which is the general equation
of geodesics) and both go though the two fixed point. They arise from
taking \(\pm\) when solving the quadratic equation to find \(t\). Both
of them separate the region where the triangle inequality fails but one
of them in between two points and the second outside.

\section{Conclusion}
\label{sec:conclusion}
In summary, in this paper, we managed to find non-degenerate
``straight lines'' in parabolic geometries, showing that parabolic
geometry possesses many non-trivial features. Also some of its properties have
been discovered and this revels its importance as ``the missing link''
between well known geometries.  We opened new options for further study
since now it is possible to create objects like triangles. This gives
opportunities to discover  the corresponding parabolic theorems
to famous ones like Pythagoras.
In other words finding the ``lines'' gives a solid footing for further
investigation into this exciting subject. Parabolic geometry is a
promising area because of the interest, refreshingly different from all other
geometries distance functions. 

The approach in this paper applies not only to \(SL_2(\mathbb{R})\)
but can be used on other semi-simple Lie groups \(G\). By considering
their action on homogeneous spaces \(G/H\) where \(H\) its subgroup,
it is possible  to create higher dimensional geometries.  

\section*{Acknowledgements}
\label{sec:aknowledgements}
I am grateful to Dr.~V.V.~Kisil for the suggestion of this topic and
useful discussions during my work on this paper. I wish to thank Cong
Chen and Richard Webb for many helpful comments. Anonymous referees
provided numerous suggestions improving the paper.

\small
\newcommand{\noopsort}[1]{} \newcommand{\printfirst}[2]{#1}
  \newcommand{\singleletter}[1]{#1} \newcommand{\switchargs}[2]{#2#1}
  \newcommand{\irm}{\textup{I}} \newcommand{\iirm}{\textup{II}}
  \newcommand{\vrm}{\textup{V}} \providecommand{\cprime}{'}
  \providecommand{\eprint}[2]{\texttt{#2}}
  \providecommand{\myeprint}[2]{\texttt{#2}}
  \providecommand{\arXiv}[1]{\myeprint{http://arXiv.org/abs/#1}{arXiv:#1}}
  \providecommand{\CPP}{\texttt{C++}} \providecommand{\NoWEB}{\texttt{noweb}}
  \providecommand{\MetaPost}{\texttt{Meta}\-\texttt{Post}}
  \providecommand{\GiNaC}{\textsf{GiNaC}}
  \providecommand{\pyGiNaC}{\textsf{pyGiNaC}}
  \providecommand{\Asymptote}{\texttt{Asymptote}}

\end{document}